\documentclass{amsart}
\usepackage{amssymb, amsmath, verbatim, amsthm,url, multirow,fullpage,mathtools,enumerate}
\usepackage{longtable, rotating,makecell}
\usepackage[aligntableaux=top]{ytableau}

\usepackage{alphalph}

\usepackage{soul}
\usepackage[colorinlistoftodos,textsize=footnotesize]{todonotes}

\setstcolor{red}

\usepackage[all]{xy}
\xyoption{poly}
\xyoption{arc}

\usetikzlibrary{decorations.pathmorphing,calc}
\usetikzlibrary{intersections}

\definecolor{light-gray}{gray}{0.6}

\tikzstyle{propagator}=[decorate,decoration={snake,amplitude=0.8mm}]
\tikzstyle{smallpropagator}=[decorate,decoration={snake,segment length=3mm,amplitude=0.5mm}]

\tikzstyle{firstdash}=[dashed,line cap=round, dash pattern=on 2pt off 1pt]
\tikzstyle{seconddash}=[dashed,line cap=round, dash pattern=on 0.5pt off 1pt]

\newcommand{\R}{\mathbb{R}}

\newcommand{\Gr}{\mathbb{G}_{\R, +}}

\def\ba #1\ea{\begin{align} #1 \end{align}}
\def\bas #1\eas{\begin{align*} #1 \end{align*}}
\def\bml #1\eml{\begin{multline} #1 \end{multline}}
\def\bmls #1\emls{\begin{multline*} #1 \end{multline*}}

\newcommand{\cI}{\mathcal{I}}

\newcommand{\II}{\mathcal{I}}

\newcommand{\cL}{\mathcal{L}}

\newcommand{\OO}{\mathcal{O}}

\newtheorem{thm}{Theorem}[section]

\newtheorem{lem}[thm]{Lemma}

\newtheorem{alg}{Algorithm}

\theoremstyle{remark}
\newtheorem{eg}[thm]{Example}

\theoremstyle{definition}

\begin{document}

\title{An algorithm to construct the Le diagram associated to a Grassmann necklace}
\author{S. Agarwala and S. Fryer}

\begin{abstract}
Le diagrams and Grassmann necklaces both index the collection of positroids in the nonnegative Grassmannian $Gr_{\geq 0}(k,n)$, but they excel at very different tasks: for example, the dimension of a positroid is easily extracted from its Le diagram, while the list of bases of a positroid is far more easily obtained from its Grassmann necklace. Explicit bijections between the two are therefore desirable. An algorithm for turning a Le diagram into a Grassmann necklace already exists; in this note we give the reverse algorithm.
\end{abstract}

\maketitle

\section{Introduction}\label{s:intro}

An element of the real Grassmannian $Gr(k,n)$ is called totally nonnegative if it has a matrix representation in which all of its maximal minors are nonnegative. This induces a cell decomposition of the nonnegative Grassmannian $Gr_{\geq 0}(k,n)$ into {\em positroids}, determined by which maximal minors are positive and which are zero. The positroids are indexed by several combinatorially interesting collections of objects, many of which first appeared in Postnikov's foundational preprint \cite{Postnikov}. In this note we focus only on two of these (Le diagrams and Grassmann necklaces, both defined in Section \ref{s:notation} below), and refer the interested reader to \cite{Postnikov} for the bigger picture.

Positroids have recently found applications in several areas of physics, notably the study of shallow water waves \cite{MR3279534} and calculating scattering amplitudes in the quantum field theory SYM $N=4$ \cite{agarwala,Amplituhedron}. Indeed, the two main current techniques used in calculating scattering amplitudes in SYM $N=4$ are BCFW recursion and Wilson loop diagrams, both of which make use of the combinatorial machinery of positroids \cite{KWZ}. In particular, one finds that the Le diagrams are crucial for understanding the geometry underlying the Wilson loop diagrams \cite{2:6paper}. In a forthcoming paper, we study the combinatorics of the positroids associated to Wilson loop diagrams, and show that the natural objects to use in this setting are the Grassmann necklaces.

Positroids also have a surprising application in noncommutative algebra, where they are closely related to the study of graded prime ideals in the quantized coordinate rings $\OO_q(M_{m,n})$ and $\OO_q(Gr(k,n))$ \cite{GLL2,MR2403308}. These ideals are indexed by {\em Cauchon diagrams} \cite{Cauchon}, which are equivalent to Le diagrams. However, obtaining a generating set for a given prime is often easier when phrased in terms of the Grassmann necklace. (This follows from \cite{MR3285618,GLL2} for $\OO_q(M_{m,n})$, and is conjectured for $\OO_q(Gr(k,n))$.)

The algorithm to construct a Le diagram from a Grassmann necklace was a result of the aforementioned forthcoming work, and we present it separately in this short note due to its usefulness for the variety of applications described above.

In Section \ref{s:notation} we set out the notation and conventions in use, and in Section \ref{s:oh alg} we give Oh's algorithm for constructing a Grassmann necklace from a Le diagram. Finally in Section \ref{s:reverse alg} we describe the inverse algorithm, and prove that it does indeed reverse the algorithm given in Section \ref{s:oh alg}.

{\bf Acknowledgements}: The authors gratefully acknowledge Karel Casteels for introducing them to Grassmann necklaces and for several helpful conversations on the subject.

\section{Notation and definitions}\label{s:notation}

Following the standard convention, we write $[n]$ for the set of integers $\{1,2,\dots, n\}$.

Given a Young diagram fitting inside a $k \times (n-k)$ box, we assign numbers to its rows and columns by arranging the numbers $[n]$ along its southeast border, starting at the north east corner. All coordinates are given in terms of these row and column labels. Notice that if we keep the same diagram but increase $n$, this changes the size of the bounding box and hence changes the row and column labels. (See Figure \ref{fig:row column numbering} for examples.) 

\begin{figure}[h!]
\[\begin{tikzpicture}[baseline=(current bounding box.east),scale=0.75]
\draw (1,0) grid (2,3);
\draw (2,1) grid (4,3);
\draw (4,2) grid (5,3);
\draw[dotted] (2,0) -- (5,0) -- (5,2);

\node at (5.15,2.5) {\footnotesize$1$};
\node at (4.5,1.85) {\footnotesize$2$};
\node at (4.15,1.5) {\footnotesize$3$};
\node at (3.5,0.85) {\footnotesize$4$};
\node at (2.5,0.85) {\footnotesize$5$};
\node at (2.15,0.5) {\footnotesize$6$};
\node at (1.5,-0.15) {\footnotesize$7$};
\end{tikzpicture}
\qquad \qquad
\begin{tikzpicture}[baseline=(current bounding box.east),scale=0.75]
\draw (1,0) grid (2,3);
\draw (2,1) grid (4,3);
\draw (4,2) grid (5,3);
\draw[dotted] (2,0) -- (6,0) -- (6,3) -- (5,3);

\node at (5.5,2.85) {\footnotesize$1$};
\node at (5.15,2.5) {\footnotesize$2$};
\node at (4.5,1.85) {\footnotesize$3$};
\node at (4.15,1.5) {\footnotesize$4$};
\node at (3.5,0.85) {\footnotesize$5$};
\node at (2.5,0.85) {\footnotesize$6$};
\node at (2.15,0.5) {\footnotesize$7$};
\node at (1.5,-0.15) {\footnotesize$8$};
\end{tikzpicture}
\]
\caption{Row and column numbering for a Young diagram with $k = 3$, $n = 7$ (left) and $k = 3$, $n = 8$ (right). The top left box in each diagram has coordinates $(1,7)$ (left diagram), $(2,8)$ (right diagram).}
\label{fig:row column numbering}
\end{figure}
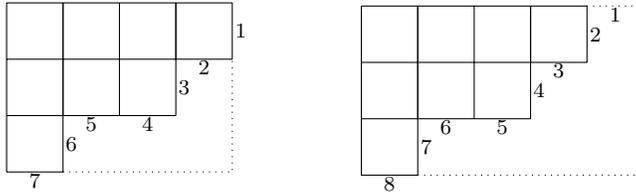

A {\bf Le diagram} of type $(k,n)$ is a Young diagram fitting inside a $k \times (n-k)$ box, together with a filling of the squares with $+$ and $0$ symbols subject to the following rule:
\begin{itemize}
\item Consider $i,j,k,l \in [n]$ with $i < k$ and $j<l$. If the boxes $(i, j)$ and $(k,l)$ both contain a $+$, then box $(k,j)$ (if it exists) must also contain a $+$.
\end{itemize}

A {\bf Grassmann necklace} of type $(k,n)$ is an ordered sequence of sets
\[\cI = (I_1, I_2, \dots, I_n),\]
with $I_i \subset [n]$ and $|I_i| = k$ for each $i$, related by the following rule:
\begin{itemize}
\item If $i \in I_i$, then $I_{i+1} = (I_i \backslash i) \cup j$ for some $j \in [n] \setminus I_i$.
\item If $i \not \in I_i$, then $I_{i+1} = I_i$.
\end{itemize}
(Note: all indices are taken mod $n$. In particular, if $i = n$ then $I_{i+1} = I_1$.)

Both of these definitions were originally stated by Postnikov in \cite[Sections 6 and 16]{Postnikov} The thread tying the definitions of Le diagrams and Grassmann necklaces together is given by the following theorem:

\begin{thm}\label{postnikovthm} \cite[Theorems 6.5 and 17.1]{Postnikov}
The positroid cells in $Gr(k,n)_{\geq 0}$ are in 1-1 correspondence with the Le diagrams of type $(k,n)$, and are also in 1-1 correspondence with the Grassmann necklaces of type $(k,n)$.
\end{thm}

\section{Oh's algorithm for obtaining a Grassmann necklace from a Le diagram}\label{s:oh alg}

In this section we recall an algorithm for obtaining a Grassmann necklace from a Le diagram, due to Oh \cite{MR2834184}.

Given two squares $(i,j)$ and $(k,l)$ in a Young diagram, we say that
\begin{itemize}
\item $(i,j)$ is {\bf strictly northwest} of $(k,l)$ if $i < k$ and $j > l$.
\item $(i,j)$ is {\bf weakly northwest} of $(k,l)$ if $i \leq k$ and $j \geq l$.
\end{itemize}

Given any square $(k,l)$ in a Le diagram $L$, it follows from the Le property that:
\begin{itemize}
\item There is either a unique nearest $+$ square strictly northwest of $(k,l)$, or all squares strictly northwest of $(k,l)$ contain a $0$.
\item There is either a unique nearest $+$ square weakly northwest of $(k,l)$, or all squares weakly northwest of $(k,l)$ contain a $0$.
\end{itemize}
See \cite{CasteelsFryer,MR2834184} for more details. Note that if $(k,l)$ contains a $+$, then the unique $+$ square weakly northwest of it is $(k,l)$ itself.



We can now state Oh's algorithm:

\begin{alg}\label{alg:oh} Let $L$ be a Le diagram of type $(k,n)$.
\begin{enumerate}
\item Label the {\em squares} along the southeast boundary of $L$ with $\overline{2}, \overline{3}, \dots, \overline{n}$, starting from the northeast corner.
\item Define $I_1 = \{\text{row labels of $L$}\}$.
\item To obtain $I_i$, $2\leq i \leq n$:
\begin{itemize}
\item Start in the nearest $+$ square weakly northwest of $\overline{i}$ (this could be $\overline{i}$ itself).
\item Step to the unique nearest $+$ square strongly northwest of the current position.
\item Continue until no more steps are possible, keeping track of the squares stepped in.
\item Set $I_i = \big(I_1 \backslash \{\text{rows involved in this path}\}\big) \cup \{\text{columns involved in this path}\}$.
\end{itemize}
\end{enumerate}
We write $\cI(L)$ for the Grassmann necklace of a Le diagram $L$.
\end{alg}
Notice that if there are {\em no} $+$ squares weakly northwest of square $\overline{i}$, then the path from $\overline{i}$ is empty and $I_i = I_1$.

\begin{eg}
In $Gr_{\geq0}(3,8)$, consider the Le diagram



\[L\ =\ \Large
\begin{tikzpicture}[baseline=(current bounding box.east),scale=0.75]
\draw (0,0) grid (2,3);
\draw (2,1) grid (4,3);
\draw (4,2) grid (5,3);
\node[gray,anchor=south west] at (4.5,2.4) {\small $\mathbf{\overline{2}}$};
\node[gray,anchor=south west] at (3.5,2.4) {\small $\mathbf{\overline{3}}$};
\node[gray,anchor=south west] at (3.5,1.4) {\small $\mathbf{\overline{4}}$};
\node[gray,anchor=south west] at (2.5,1.4) {\small $\mathbf{\overline{5}}$};
\node[gray,anchor=south west] at (1.5,1.4) {\small $\mathbf{\overline{6}}$};
\node[gray,anchor=south west] at (1.5,0.4) {\small $\mathbf{\overline{7}}$};
\node[gray,anchor=south west] at (0.5,0.4) {\small $\mathbf{\overline{8}}$};

\node at (0.5,0.5) {$0$};
\node at (0.5,1.5) {$+$};
\node at (0.5,2.5) {$+$};
\node at (1.5,0.5) {$+$};
\node at (1.5,1.5) {$+$};
\node at (1.5,2.5) {$+$};
\node at (2.5,1.5) {$0$};
\node at (2.5,2.5) {$0$};
\node at (3.5,1.5) {$+$};
\node at (3.5,2.5) {$0$};
\node at (4.5,2.5) {$+$};

\node at (5.15,2.5) {\footnotesize$1$};
\node at (4.5,1.85) {\footnotesize$2$};
\node at (4.15,1.5) {\footnotesize$3$};
\node at (3.5,0.85) {\footnotesize$4$};
\node at (2.5,0.85) {\footnotesize$5$};
\node at (2.15,0.5) {\footnotesize$6$};
\node at (1.5,-0.15) {\footnotesize$7$};
\node at (0.5,-0.15) {\footnotesize$8$};

\end{tikzpicture},
\]
where we have already labelled the boundary squares as described in Algorithm \ref{alg:oh}.

To compute $I_4$, for example: the sequence of steps from the square labelled $\overline{4}$ is $(3,4), (1,7)$, and so
\[I_4 = \big(\{1,3,6\}\backslash \{1,3\}\big) \cup \{4,7\} = 467.\]
The complete Grassmann necklace of this diagram is
\[\II(L) = (136,236,367,467,678,678,178,168).\]

\end{eg}

\section{Reversing Oh's algorithm}\label{s:reverse alg}

Oh's algorithm for converting a Le diagram into a Grassmann necklace is an extremely useful tool for computing concrete examples, and for translating between different subject areas (e.g. from noncommutative algebra to combinatorics, as in \cite{CasteelsFryer}). However, it is also useful to be able to translate a Grassmann necklace into a Le diagram (e.g. to quickly compute the dimension of the associated positroid), but no algorithm exists in the literature to perform this process easily.

One can of course slowly construct the Le diagram from the northeast corner in order to ensure it produces the right Grassmann necklace, but for large examples this can be slow and prone to error. Instead we present the following streamlined version of this process:
\begin{alg}\label{alg:susama}
Let $\cI$ be a Grassmann necklace of type $(k,n)$.
\begin{enumerate}
\item In a $k\times (n-k)$ box, draw the Young diagram whose rows are labelled by $I_1$.
\item For each $i$, $2 \leq i \leq n$:
\begin{itemize}
\item Write
\[I_1 \backslash I_i = \{a_1 > a_ 2 > \dots > a_r\}, \quad I_i \backslash I_1 = \{b_1 < b_2 < \dots < b_r\}.\]
\item Place a $+$ label in each of the squares
\[(a_1,b_1), (a_2,b_2), \dots, (a_r,b_r).\]
\end{itemize}
\item After performing step (2) for all $i$, place a $0$ in any remaining unfilled boxes.
\end{enumerate}
Write $\cL(\cI)$ for the diagram obtained from a Grassmann necklace $\cI$ via this process.
\end{alg}

\begin{eg}
Consider the Grassmann necklace

\begin{equation}\label{eq:GN for example}\cI = (1247, 2347, 3478,4678,5678,4678,1478,1478).\end{equation}

This necklace has 8 terms, each containing 4 entries, so it is of type $(4,8)$. To construct $\cL(\cI)$, we start by drawing the Young diagram with rows labelled by $I_1$ inside a $4 \times (8-4)$ box:

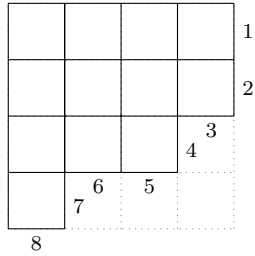
\[\begin{tikzpicture}[baseline=(current bounding box.east),scale=0.75]
\draw[gray,dotted] (0,0) grid (4,4);
\draw (0,0) grid (1,4) (1,1) grid (3,4) (3,2) grid (4,4);
\node at (4.25,3.5) {\footnotesize $1$};
\node at (4.25,2.5) {\footnotesize $2$};
\node at (3.6,1.75) {\footnotesize $3$};
\node at (3.25,1.4) {\footnotesize $4$};
\node at (2.5,0.75) {\footnotesize $5$};
\node at (1.6,0.75) {\footnotesize $6$};
\node at (1.25,0.4) {\footnotesize $7$};
\node at (0.5,-0.25) {\footnotesize $8$};
\end{tikzpicture}
\]

The squares which should contain a $+$ label are specified by Algorithm \ref{alg:susama}:

\[\begin{array}{c|c|c|c}
\ \ i \ \  & I_1 \backslash I_i & I_i \backslash I_1 & + \text{ squares } \\ \hline
2 & 1 & 3  & (1,3) \\
3 & 2, 1  & 3, 8  & (2,3), (1,8)  \\
4 & 2, 1 & 6, 8  & (2,6), (1,8) \\
5 & \ \ 4, 2, 1 \ \  & \ \ 5, 6, 8 \ \  & (4,5), (2,6), (1,8)  \\
6 & 2, 1  & 6, 8  & (2,6), (1,8)  \\
7 & 2  & 8  & (2,8)  \\
8 & 2  & 8  & (2,8)  \\
\end{array}\]
\vspace{0.5em}

Placing $+$ labels in these squares and $0$ labels in all remaining squares, we obtain the Le diagram
\[\cL(\cI)\   = \  \Large
\begin{tikzpicture}[baseline=(current bounding box.east),scale=0.75]
\draw (0,0) grid (1,4) (1,1) grid (3,4) (3,2) grid (4,4);
\node at (4.25,3.5) {\footnotesize $1$};
\node at (4.25,2.5) {\footnotesize $2$};
\node at (3.6,1.75) {\footnotesize $3$};
\node at (3.25,1.4) {\footnotesize $4$};
\node at (2.5,0.75) {\footnotesize $5$};
\node at (1.6,0.75) {\footnotesize $6$};
\node at (1.25,0.4) {\footnotesize $7$};
\node at (0.5,-0.25) {\footnotesize $8$};

\node at (0.5,0.5) {$0$};
\node at (0.5,1.5) {$0$};
\node at (0.5,2.5) {$+$};
\node at (0.5,3.5) {$+$};
\node at (1.5,1.5) {$0$};
\node at (1.5,2.5) {$+$};
\node at (1.5,3.5) {$0$};
\node at (2.5,1.5) {$+$};
\node at (2.5,2.5) {$0$};
\node at (2.5,3.5) {$0$};
\node at (3.5,2.5) {$+$};
\node at (3.5,3.5) {$+$};

\end{tikzpicture}\]
The reader is invited to verify that applying Algorithm \ref{alg:oh} to $\cL(\cI)$ above yields exactly the Grassmann necklace in \eqref{eq:GN for example}.
\end{eg}

To prove that Algorithm \ref{alg:susama} is indeed the reverse of Algorithm \ref{alg:oh}, we show that the lists of squares constructed in step 2 of Algorithm \ref{alg:susama} are precisely the squares appearing in the paths defined by Algorithm \ref{alg:oh} (Lemma \ref{res:enough pluses}), and that every $+$ square in a Le diagram must appear in one of these paths (Lemma \ref{res:every plus appears in a GN}).

\begin{lem}\label{res:enough pluses}
Given a Le diagram $L$, every $+$ square in $\cL(\cI(L))$ is also a $+$ square in $L$.
\end{lem}
\begin{proof}
Fix $i \in [n] \backslash \{1\}$, and recall that the $i$th term in the Grassmann necklace $\cI(L)$ is defined by
\[I_i = \big(I_1 \backslash \{ \text{rows involved in the path from } \overline{i}\}\big) \cup \{ \text{columns involved in the path from }\overline{i} \},\]
where ``the path from $\overline{i}$'' refers to the sequence of steps defined by Algorithm \ref{alg:oh} and starting from square $\overline{i}$ in $L$.

Once we have $L$ and $\cI(L)$, we can describe the path from $\overline{i}$ in terms of $I_1$ and $I_i$ as follows. The rows involved in the path from $\overline{i}$ are:
\begin{align*}
\{\text{rows in the path from } \overline{i}\} & = \{ \text{all row labels}\} \cap \{\text{rows in path from }\overline{i} \}\\
& = I_1 \backslash \{\text{rows not in path from }\overline{i} \} \\
& = I_1 \backslash (I_1 \cap I_i) \\
& = I_1 \backslash I_i.
\end{align*}
Since the path is moving strictly north at each step, the row labels in $I_1 \backslash I_i$ should be ordered from maximum to minimum. Similarly, the columns involved in the path from square $\overline{i}$ are given by
\begin{align*}
\{\text{columns in the path from $\overline{i}$}\} & = \{ \text{all column labels}\} \cap \{ \text{columns in path from }\overline{i} \} \\
& = \{I_1^c\} \cap \{I_1^c\cap I_i \}\\
& = I_i \backslash I_1.
\end{align*}
The path is moving strictly west at each step, so the column labels in $I_i \backslash I_1$ should be ordered from minimum to maximum.

Thus the squares which acquire a $+$ label when constructing $\cL(\cI(L))$ are exactly those which are involved in at least one path in $L$, each of which must be a $+$ square in $L$ by the definition of Algorithm \ref{alg:oh}.
\end{proof}

It is not immediately clear from Algorithm \ref{alg:oh} that {\em every} $+$ square in $L$ contributes to at least one term in the Grassmann necklace, i.e. $\cL(\cI(L))$ may contain a strict subset of the $+$ squares in $L$. The next lemma verifies that this is not the case.

\begin{lem}\label{res:every plus appears in a GN}
Given a Le diagram $L$, every $+$ square in $L$ is also a $+$ square in $\cL(\cI(L))$.
\end{lem}
\begin{proof}
Suppose for contradiction that there is a $+$ square (which we will call $+_a$) in $L$ which appears in none of the paths from Algorithm \ref{alg:oh}. We claim that in this case, there must exist another $+$ square in $L$ which is strictly southeast of $+_a$ and also does not appear in any of the paths from Algorithm \ref{alg:oh}.

First note that $+_a$ cannot be the nearest $+$ square weakly northwest of any boundary square, and in particular it cannot be a boundary square itself; this is because all such squares appear in step one of Algorithm \ref{alg:oh}. This guarantees that there is at least one $+$ square to the east of $+_a$ (in the same row), and at least one $+$ square to the south of $+_a$ (in the same column).

Thus we have the following type of configuration in $L$ (bearing in mind that either or both of the blocks of 0s could be empty):

\begin{equation}\label{eq:Le proof 1}
\begin{tikzpicture}[baseline=(current bounding box.east),scale = 0.5]
\draw (0,0) rectangle (1,1);
\node at (0.5,0.5) {\large $+$};

\draw (0,1) rectangle (1,3);
\node at (0.5,2) {$0$};

\draw (0,3) rectangle (1,4);
\node at (0.5,3.5) {\large $+_a$};

\draw (1,3) rectangle (4,4);
\node at (2.5,3.5) {$0$};

\draw (4,3) rectangle (5,4);
\node at (4.5,3.5) {\large $+$};



\draw[gray,dotted] (4,3) -- (4,1) (5,3) -- (5,1);
\draw[gray,dotted] (1,1) -- (4,1) (1,0) -- (4,0);

\draw[gray] (4,0) rectangle (5,1);
\node[gray] at (4.5,0.5) {$*$};

\draw (1,4) -- (1,6) (4,4) -- (4,6);
\node at (2.5, 5) {\Small all 0s $\uparrow$};

\draw (0,1) -- (-2,1) (0,3) -- (-2,3);
\node at (-1,2) { $\substack{\leftarrow \\ \text{all 0s}}$};

\end{tikzpicture}
\end{equation}

By the Le condition, if the square labelled $*$ is inside the boundary of $L$ then it must contain a $+$, and we are done. If not, we have the following type of configuration (the dashed line indicates the diagram boundary):

\begin{equation}\label{eq:Le proof 2}
\begin{tikzpicture}[baseline=(current bounding box.east),scale = 0.5]

\draw[gray, fill = gray, opacity = 0.5] (1,1) rectangle (4,3);

\draw (0,0) rectangle (1,1);
\node at (0.5,0.5) {\large $+$};

\draw (0,1) rectangle (1,3);
\node at (0.5,2) {$0$};

\draw (0,3) rectangle (1,4);
\node at (0.5,3.5) {\large $+_a$};

\draw (1,3) rectangle (4,4);
\node at (2.5,3.5) {$0$};

\draw (4,3) rectangle (5,4);
\node at (4.5,3.5) {\large $+$};


\draw[dashed,gray] (4,1) -- ++(1,0) -- ++(0,1) -- ++(1,0) -- ++(0,1);
\draw[dashed,gray] (4,1)  -- ++(0,-1) -- ++(-1,0) -- ++(0,-1)-- ++(-1,0);

\draw (1,4) -- (1,6) (4,4) -- (4,6);
\node at (2.5, 5) {\Small all 0s $\uparrow$};

\draw (0,1) -- (-2,1) (0,3) -- (-2,3);
\node at (-1,2) { $\substack{\leftarrow \\ \text{all 0s}}$};

\end{tikzpicture}
\end{equation}

Note that the shaded region in \eqref{eq:Le proof 2} {\em must} have positive width and height, because otherwise $+_a$ is the nearest $+$ square weakly northwest of a boundary square. By the same argument, the shaded region in \eqref{eq:Le proof 2} must contain at least one $+$ square.


In either case, we have demonstrated the existence of a square $+_b$ strictly southeast of $+_a$, and with the pattern of 0 squares indicated in diagrams \eqref{eq:Le proof 1} or \eqref{eq:Le proof 2}. This pattern of 0s is forced by the Le condition, and it guarantees that any path which steps in $+_b$ must also step in $+_a$ (possibly with some intermediate steps). Therefore $+_b$ does not appear in any of the paths from Algorithm \ref{alg:oh} either. This argument can be repeated indefinitely, a contradiction to the fact that the diagram is finite.

It follows that every $+$ square in $L$ contributes to at least one term in the Grassmann necklace $\cI(L)$, and hence by Lemma \ref{res:enough pluses} is also a $+$ square in $\cL(\cI(L))$.
\end{proof}

Thus we arrive at the main conclusion of the paper.

\begin{thm}
Algorithm \ref{alg:susama} uniquely constructs the Le diagram associated to a Grassmann necklace.
\end{thm}
\begin{proof}
This follows from Lemmas \ref{res:enough pluses} and Lemma \ref{res:every plus appears in a GN}, and the fact that Algorithm \ref{alg:oh} uniquely constructs the Grassmann necklace associated to a Le diagram.
\end{proof}


\end{document}